\documentclass[a4paper,oneside,final]{amsart}
\usepackage[utf8]{inputenc}
\usepackage[a4paper,margin=3cm]{geometry} 
\usepackage{bm}
\usepackage{amsmath}
\usepackage{amsthm}
\usepackage{amscd}
\usepackage{amssymb}
\usepackage{MnSymbol}
\usepackage{latexsym}
\usepackage{eucal}
\usepackage{dsfont}
\usepackage{mathtools}
\usepackage{enumitem}
\usepackage[notref,notcite]{showkeys}

\usepackage[colorlinks,pdftex]{hyperref}
\hypersetup{
linkcolor=black,
citecolor=black,
pdftitle={},
pdfauthor={Franziska K\"uhn},
pdfkeywords={},
}

\makeatletter
\renewcommand\section{\@startsection{section}{1}{0mm}{-1.5\baselineskip}{\baselineskip}{\normalsize\bfseries\sffamily}}
\renewcommand\subsection{\@startsection{subsection}{1}{0mm}{-\baselineskip}{\baselineskip}{\normalsize\bfseries\sffamily}}
\makeatother

\makeatletter
\def\@fnsymbol#1{\ensuremath{\ifcase#1\or *\or **\or \dagger\or \ddagger\or
   \mathsection\or \mathparagraph\or \|\or \dagger\dagger
   \or \ddagger\ddagger \else\@ctrerr\fi}}

\newlength{\preskip}
\newlength{\postskip}
\makeatother

\newtheoremstyle{theorem}{\preskip}{\postskip}{\itshape}{}{\bfseries}{}
{.5em}{\textbf{\thmname{#1}\thmnumber{ #2} (\thmnote{ #3})}}
\newtheoremstyle{definition}{\preskip}{\postskip}{\normalfont}{0pt}{\bfseries}{}{.5em}{}
\newtheoremstyle{remark}{\preskip}{\postskip}{\normalfont}{0pt}{\bfseries}{}{.5em}{}

\swapnumbers
\theoremstyle{theorem} \newtheorem{thm}{Theorem}[section]
\theoremstyle{theorem} \newtheorem{lem}[thm]{Lemma}
\theoremstyle{theorem} \newtheorem{prop}[thm]{Proposition}
\theoremstyle{theorem} \newtheorem{kor}[thm]{Corollary}
\theoremstyle{definition} 
\theoremstyle{remark} 
\theoremstyle{remark} 
\theoremstyle{definition} 
\theoremstyle{definition} \newtheorem*{ack}{Acknowledgements}
\theoremstyle{remark} \newtheorem{bem}[thm]{Remark}
\theoremstyle{remark} 
\theoremstyle{definition}  \newtheorem{bsp}[thm]{Example}
\theoremstyle{definition}  
\theoremstyle{definition}

\DeclareMathOperator \re {Re}
\DeclareMathOperator \im {Im}

\DeclareMathOperator \id {id}

\DeclareMathOperator \lip {Lip}
\DeclareMathOperator \spt {supp}

\DeclareMathOperator \tr {tr}

\newcommand{\I}{\mathds{1}}

\newcommand\fa{\qquad \text{for all \ }}
\newcommand{\firstpara}[1]{ \ \textup{#1\ \ }}
\newcommand{\para}[1]{\bigskip\textup{#1\ \ }}

\newcommand\mc[1] {\mathcal{#1}}
\newcommand\mbb[1] {\mathds{#1}}

\newcommand{\eps}{\varepsilon}

\begin{document}

\title[Schauder estimates for equations associated with L\'evy generators]{Schauder estimates for equations associated with L\'evy generators}
\author[F.~K\"{u}hn]{Franziska K\"{u}hn} 
\address[F.~K\"{u}hn]{Institut de Math\'ematiques de Toulouse, Universit\'e Paul Sabatier III Toulouse, 118 Route de Narbonne, 31062 Toulouse, France. \emph{On leave from:} TU Dresden, Fachrichtung Mathematik, Institut f\"{u}r Mathematische Stochastik, 01062 Dresden, Germany.}
\email{franziska.kuhn@math.univ-toulouse.fr}

\begin{abstract}
	We study the regularity of solutions to the integro-differential equation $Af-\lambda f=g$ associated with the infinitesimal generator $A$ of  a L\'evy process. We show that gradient estimates for the transition density can be used to derive Schauder estimates for $f$. Our main result allows us to establish Schauder estimates for a wide class of L\'evy generators, including generators of stable L\'evy processes and subordinated Brownian motions. Moreover, we obtain new insights on the (domain of the) infinitesimal generator of a L\'evy process whose characteristic exponent $\psi$ satisfies $\text{Re} \, \psi(\xi) \asymp |\xi|^{\alpha}$ for large $|\xi|$. We discuss the optimality of our results by studying in detail the domain of the infinitesimal generator of the Cauchy process.
\end{abstract}
\maketitle

\section{Introduction} \label{intro}

Let $(X_t)_{t \geq 0}$ be a L\'evy process. By the L\'evy--Khintchine formula, the infinitesimal generator $(A,\mc{D}(A))$ of $(X_t)_{t \geq 0}$ has the representation \begin{equation*}
	Af(x) = b \cdot \nabla f(x) + \frac{1}{2} \tr(Q \cdot \nabla^2 f(x)) + \int_{y \neq 0} \left( f(x+y)-f(x)-\nabla f(x) \cdot y \I_{(0,1)}(|y|) \right) \, \nu(dy) 
\end{equation*}
for smooth compactly supported functions $f \in C_c^{\infty}(\mbb{R}^d)$ where $(b,Q,\nu)$ is the L\'evy triplet of $(X_t)_{t \geq 0}$, cf.\ Section~\ref{def}. In this paper, we study the H\"{o}lder regularity of solutions $f \in \mc{D}(A)$ to the integro-differential equation \begin{equation}
	Af + \varrho f= g  \label{intro-eq3}
\end{equation}
for fixed $\varrho \in \mbb{R}$. We are interested in the following question: If $g$ is $\beta$-H\"{o}lder continuous for some $\beta \geq 0$, then what can we say about the regularity of $f$? In particular: How regular is a function $f \in \mc{D}(A)$? \par
For the particular case that $A$ is a second order differential operator, i.\,e.\ $\nu=0$, the regularity of solutions to \eqref{intro-eq3} is well understood, see e.\,g.\ \cite{gilbarg}, and therefore our focus is on non-local L\'evy generators. An important example of a non-local L\'evy generator is the fractional Laplacian \begin{equation*}
	-(-\Delta)^{\alpha/2} f(x) = c_{d,\alpha} \int_{y \neq 0} \left( f(x+y)-f(x)-\nabla f(x) \cdot y \I_{(0,1)}(|y|) \right) \frac{1}{|y|^{d+\alpha}} \, dy, \quad f\in C_c^{\infty}(\mbb{R}^d),
\end{equation*} 
which is the infinitesimal generator of the isotropic $\alpha$-stable L\'evy process, $\alpha \in (0,2)$, and which plays an important role in analysis and probability theory, see e.\,g.\ the survey paper \cite{kwas17} for further information. Bass \cite{bass08} showed that the solution to $- (-\Delta)^{\alpha/2} f=g$ satisfies the Schauder estimate \begin{equation*}
	\|f\|_{C^{\alpha+\beta}_b(\mbb{R}^d)} \leq L \left(\|f\|_{\infty} + \|g\|_{C^{\beta}_b(\mbb{R}^d)}\right)
\end{equation*}
for $\beta \geq 0$ such that neither $\beta$ nor $\alpha+\beta$ are integers. More recently, Ros-Oton \& Serra \cite{ros-oton16} established Schauder estimates for solutions to \eqref{intro-eq3} for generators of symmetric stable L\'evy processes. Bae \& Kassmann \cite{bae} introduced generalized H\"{o}lder space and studied, in particular, the regularity of solutions for L\'evy operators of the form \begin{equation*}
	Af(x) = \int_{y \neq 0} \left( f(x+y)-f(x)-\nabla f(x) \cdot y \I_{(0,1)}(|y|) \right) \frac{1}{|y|^d \varphi(|y|)} \, dy
\end{equation*}
where $\varphi: \mbb{R}^d \to (0,\infty)$ is a ``nice'' function.  Furthermore, it is known that the classical theory for pseudo-differential operators can be used to study the regularity of solutions to \eqref{intro-eq3} if the characteristic exponent $\psi$ of $(X_t)_{t \geq 0}$ is sufficiently smooth, see  \cite{jac2,stein93}. Since \begin{equation*}
	\psi \in C^{2n}(\mbb{R}^d) \iff \int_{|y|>1} |y|^{2n} \, \nu(dy) < \infty \iff \mbb{E}(|X_1|^{2n})< \infty, \quad n \in \mbb{N},
\end{equation*}
cf.\ \cite{moments,sato}, this approach excludes many interesting examples of L\'evy processes which do not have moments of sufficiently high order. Let us mention that the questions, which we discuss in this paper, are also related to the regularity of harmonic functions: If $g=0$ and $\varrho=0$ in \eqref{intro-eq3}, i.\,e.\ $Af=0$, then $f$ is harmonic for $A$, and there is an extensive literature on the regularity of functions which are harmonic for a L\'evy generator, cf.\ \cite{grz14,hansen17,ryznar15,szt10} and the references therein. The regularity of solutions to elliptic integro-differential equations $Af=g$ has been studied, more generally, for classes of L\'evy-type operators, see e.\,g.\ \cite{bae,bass08,dong12,jac2,reg-feller}, and for non-linear integro-differential operators, greatly influenced by the works of Barles et.\ al \cite{barles10} and Caffarelli \& Silvestre \cite{caff}. 

The approach, which we follow in this paper, relies on regularizing properties of the resolvent $R_{\lambda}$ associated with the L\'evy process $(X_t)_{t \geq 0}$, \begin{equation*}
	R_{\lambda} h(x) := \int_{(0,\infty)} e^{-\lambda t} \mbb{E}h(x+X_t) \, dt, \qquad \lambda>0, h \in \mc{B}_b(\mbb{R}^d), x \in \mbb{R}^d.
\end{equation*}
The main idea is to use gradient estimates for the transition density $p_t$ of $X_t$ to measure the regularizing effect of $R_{\lambda}$. More precisely, we will show that the gradient estimate \begin{equation*}
	\int_{\mbb{R}^d} |\nabla p_t(x)| \, dx \leq M t^{-1/\alpha}, \quad t \in (0,T], 
\end{equation*}
implies that $R_{\lambda}$ has a regularizing effect of order $\alpha$, i.\,e.\ \begin{equation*}
	h \in \mc{C}_b^{\delta}(\mbb{R}^d) \implies R_{\lambda} h \in \mc{C}_b^{\alpha+\delta}(\mbb{R}^d)
\end{equation*}
for any $\delta \geq 0$, cf.\ Section~\ref{def} for the definition of the H\"{o}lder--Zygmund spaces $\mc{C}_b^{\gamma}(\mbb{R}^d)$. As $\mc{D}(A) = R_{\lambda}(C_{\infty}(\mbb{R}^d))$ this gives, in particular, $\mc{D}(A) \subseteq \mc{C}_b^{\alpha}(\mbb{R}^d)$. Our main result, Theorem~\ref{levy-9}, shows that, more generally, the implication \begin{equation*}
	Af \in \mc{C}_b^{\delta}(\mbb{R}^d) \implies f \in \mc{C}_b^{\alpha+\delta}(\mbb{R}^d) 
\end{equation*}
holds for any $\delta \geq 0$. 

\begin{thm} \label{levy-9}
	Let $(X_t)_{t \geq 0}$ be a L\'evy process with infinitesimal generator $(A,\mc{D}(A))$ and characteristic exponent $\psi: \mbb{R}^d \to \mbb{C}$ satisfying the Hartman--Wintner condition \begin{equation}
		\lim_{|\xi| \to \infty} \frac{\re \psi(\xi)}{\log(|\xi|)} = \infty. \label{hw} \tag{HW}
	\end{equation}
	 Assume that there exist constants $M>0$, $T>0$ and $\alpha \in (0,2]$ such that the transition density $p_t$ of $(X_t)_{t \geq 0}$ satisfies \begin{equation}
		\int_{\mbb{R}^d} |\nabla p_t(x)| \, dx \leq M t^{-1/\alpha}, \qquad t \in (0,T]. \label{levy-eq1}
	\end{equation}
	If $f \in \mc{D}(A)$ is such that  \begin{equation*}
		Af + \varrho f= g \in \mc{C}^{\delta}_b(\mbb{R}^d) 
	\end{equation*}
	for some $\delta \geq 0$ and $\varrho \in \mbb{R}$, then $f \in \mc{C}_b^{\alpha+\delta}(\mbb{R}^d)$ and the Schauder estimate\begin{equation}
		\|f\|_{\mc{C}_b^{\alpha+\delta}(\mbb{R}^d)} \leq L \big(\|g\|_{\mc{C}^{\delta}_b(\mbb{R}^d)} + \|f\|_{\infty}\big) \label{levy-eq33}
	\end{equation}
	holds for a finite constant $L=L(\varrho,\delta,\alpha,M,d,T)$. In particular, $\mc{D}(A) \subseteq \mc{C}_{\infty}^{\alpha}(\mbb{R}^d)$.
\end{thm}

\begin{bem} \label{levy-1}\begin{enumerate}[wide, labelwidth=!, labelindent=0pt]
	\item\label{levy-1-i} From the proof of Theorem~\ref{levy-9} it is possible to obtain an explicit expression for the constant $L$ in terms of $\delta$, $\alpha$, $M$, $d$ and $T$. 
	\item\label{levy-1-ii} Condition \eqref{levy-eq1} is equivalent to saying that the semigroup $P_t u(x) := \mbb{E}u(x+X_t)$ satisfies the gradient estimate \begin{equation*}
		\|\nabla P_t u\|_{\infty} \leq M' t^{-1/\alpha} \|u\|_{\infty}, \qquad t \in (0,T), \, u \in \mc{B}_b(\mbb{R}^d),
	\end{equation*}
	cf.\ \cite[Lemma 4.1]{euler-maruyama} for details.
	\item\label{levy-1-iii} It is no restriction to assume that $\alpha \leq 2$. If \eqref{levy-eq1} holds for some $\alpha \geq 0$, then $\alpha \leq 2$, cf.\ Remark~\ref{levy-5}\eqref{levy-5-ii}.
	\item\label{levy-1-iv} The Hartman--Wintner condition \eqref{hw} ensures that $X_t$ has a smooth density $p_t$ for all $t>0$, see \cite{knop13} for a thorough discussion of \eqref{hw}. 
\end{enumerate} \end{bem}

Gradient estimates for L\'evy processes have been  intensively studied in the last years, e.\,g.\ \cite{szcz17,kaleta15,matters,ryznar15,ssw12} to mention but a few, and therefore Theorem~\ref{levy-9} applies to a wide class of L\'evy processes.  If $(X_t)_{t \geq 0}$ is a subordinated Brownian motion, then it is possible to  derive gradient estimates from heat kernel estimates for the transition density using the dimension walk formula, cf.\ \cite[Corollary 3.2]{schoenberg}. \par
Theorem~\ref{levy-9} will be proved in Section~\ref{proof}, and in Section~\ref{ex} we will illustrate Theorem~\ref{levy-9} with some examples and applications. In particular, we will present Schauder estimates for elliptic equations $Af+\varrho f=g$ associated with generators of continuous L\'evy processes, stable L\'evy processes and subordinated Brownian motions.  Moreover, we will study in detail the infinitesimal generator $(A,\mc{D}(A))$ of a L\'evy process whose characteristic exponent $\psi$ satisfies the sector condition, $|\im \psi(\xi)| \leq c |\re \psi(\xi)|$, and 
\begin{equation*}
	\re \psi(\xi) \asymp |\xi|^{\alpha} \quad \text{as $|\xi| \to \infty$};
\end{equation*}
combining Theorem~\ref{levy-9} with results from \cite{ihke,ssw12} we will show that  \begin{equation}
	\mc{C}_{\infty}^{\alpha+}(\mbb{R}^d) 
	:= \bigcup_{\beta>\alpha} \mc{C}_{\infty}^{\beta}(\mbb{R}^d) 
	\subseteq \mc{D}(A) 
	\subseteq \mc{C}_{\infty}^{\alpha}(\mbb{R}^d), \label{intro-eq11}
\end{equation} 
and this, in turn, will allow us to prove that $\mc{D}(A)$ is an algebra, that is $f \cdot g \in \mc{D}(A)$ for any $f,g \in \mc{D}(A)$, and that \begin{equation*}
	A(f \cdot g) = g \cdot Af + f \cdot Ag + \Gamma(f,g), \qquad f,g \in \mc{D}(A)
\end{equation*}
where $\Gamma$ is the Carr\'e du Champ operator, cf.\ Theorem~\ref{levy-11}. It is natural to ask whether the inclusions in \eqref{intro-eq11} are strict and whether \eqref{intro-eq11} is the optimal way to describe $\mc{D}(A)$ in terms of H\"older spaces. In Section~\ref{cauchy} we will investigate these questions for the case $\alpha=1$, which is of particular interest since there is no canonical way to define the H\"{o}lder space $\mc{C}^1(\mbb{R}^d)$. We will show for the two-dimensional Cauchy process that \eqref{intro-eq11} (with $\alpha=1$) is indeed the best possible way to describe the domain in terms of H\"older spaces and, moreover, we will see that the inclusions are strict.

\section{Basic definitions and notation} \label{def}

We consider the Euclidean space $\mbb{R}^d$ with the canonical scalar product $x \cdot y := \sum_{j=1}^d x_j y_j$ and the Borel $\sigma$-algebra $\mc{B}(\mbb{R}^d)$ generated by the open balls $B(x,r)$. For functions $f,g: \mbb{R}^d \to (0,\infty)$ we write $f \asymp g$ as $|x| \to \infty$ if there exist constants $c>0$ and $R>0$ such that \begin{equation*}
	\frac{1}{c} f(x) \leq g(x) \leq c f(x) \fa |x| \geq R.
\end{equation*}
If $f$ is a real-valued function, then $\spt f$ denotes its support, $\nabla f$ the gradient and $\nabla^2 f$ the Hessian of $f$. For $\alpha \geq 0$ we set \begin{equation}
	\lfloor \alpha \rfloor := \max\{k \in \mbb{N}_0; k \leq \alpha\} \quad \text{and} \quad \lsem \alpha \rsem := \max\{k \in \mbb{N}_0; k < \alpha\}. \label{def-eq1}
\end{equation}

\emph{Function spaces:} $\mc{B}_b(\mbb{R}^d)$ is the space of bounded Borel-measurable functions $f: \mbb{R}^d \to \mbb{R}$. The smooth functions with compact support are denoted by $C_c^{\infty}(\mbb{R}^d)$, and $C_{\infty}(\mbb{R}^d)$ is the space of continuous functions $f: \mbb{R}^d \to \mbb{R}$ vanishing at infinity. Superscripts $k\in\mbb{N}$ are used to denote the order of differentiability, e.\,g.\ $f \in C_{\infty}^k(\mbb{R}^d)$ means that $f$ and its derivatives up to order $k$ are $C_{\infty}(\mbb{R}^d)$-functions. For $\alpha \geq 0$ we define H\"older--Zygmund spaces $\mc{C}_b^{\alpha}(\mbb{R}^d)$ by \begin{equation}
	\mc{C}_b^{\alpha}(\mbb{R}^d) := \left\{f \in C_b(\mbb{R}^d); \|f\|_{\mc{C}_b^{\alpha}(\mbb{R}^d)} := \|f\|_{\infty}+\sup_{x \in \mbb{R}^d,  h \neq 0} \frac{|\Delta_h^{\lfloor \alpha \rfloor +1} f(x)|}{|h|^{\alpha}} < \infty \right\} \label{def-eq2}
\end{equation}
where \begin{equation*}
	\Delta_h f(x) := f(x+h) - f(x) \qquad \Delta_h^j f(x) := \Delta_h (\Delta_h^{j-1} f)(x), \quad j \geq 2
\end{equation*}
are iterated difference operators. Moreover, we set \begin{equation}
	\mc{C}_{\infty}^{\alpha}(\mbb{R}^d) := \mc{C}_b^{\alpha}(\mbb{R}^d) \cap C_{\infty}^{\lsem \alpha \rsem}(\mbb{R}^d) \quad \text{and} \quad \mc{C}_{\infty}^{\alpha+}(\mbb{R}^d) := \bigcup_{\eps>0} \mc{C}_{\infty}^{\alpha+\eps}(\mbb{R}^d). \label{def-eq25}
\end{equation}
For $\alpha \in (0,\infty) \backslash \mbb{N}$ the H\"older space $\mc{C}_b^{\alpha}(\mbb{R}^d)$ coincides with the ``classical'' H\"older space $C_b^{\alpha}(\mbb{R}^d)$ equipped with norm
	\begin{equation*}
	\|f\|_{C_b^{\alpha}(\mbb{R}^d)} := \|f\|_{\infty}+\sum_{j=0}^{\lfloor \alpha \rfloor}
	\sum_{\substack{\beta \in \mbb{N}_0^d \\ |\beta| = j}} \|\partial^{\beta} f\|_{\infty}
	+   \max_{\substack{\beta \in \mbb{N}_0^d \\ |\beta| = \lfloor \alpha \rfloor}}   \sup_{x \neq y} \frac{|\partial^{\beta} f(x)-\partial^{\beta} f(y)|}{|x-y|^{\alpha-\lfloor \alpha \rfloor}}.
\end{equation*}
If $\alpha \in \mbb{N}$ is an integer, then the H\"older--Zygmund space $\mc{C}_b^{\alpha}(\mbb{R}^d)$ is strictly larger than $C_b^{\alpha}(\mbb{R}^d)$. For $\alpha=1$ it is possible to show that $\mc{C}_{b}^1(\mbb{R}^d)$ is strictly larger than the space of bounded Lipschitz continuous functions, cf.\ \cite[p.~ 148]{stein}, which is, in turn, strictly larger than $C_{b}^1(\mbb{R}^d)$. By \cite[Theorem 2.7.2.2]{triebel78}, it holds for all $\alpha>0$ that \begin{equation}
	\|f\|_{\mc{C}_b^{\alpha}(\mbb{R}^d)} \asymp \|f\|_{C_b^{\ell}(\mbb{R}^d)}+\sum_{\substack{\beta \in \mbb{N}_0^d \\ |\beta| \leq \ell}} \sup_{\substack{x \in \mbb{R}^d \\ 0<|h|<\delta}} \frac{|\Delta_h^{k} \partial_x^{\beta} f(x)|}{|h|^{\alpha-\ell}}\label{def-eq3}
\end{equation}
for any $\delta \in (0,\infty]$ and $k,\ell \in \mbb{N}_0$ such that $\ell<\alpha$ and $\ell+k>\alpha$. Later on, we will use the following result from interpolation theory. If $T: C_b(\mbb{R}^d) \to C_b(\mbb{R}^d)$ is a linear operator, then \begin{equation}
	\|T\|_{\mc{C}_b^{\lambda \alpha_1+(1-\lambda) \alpha_2} \to \mc{C}_b^{\lambda \beta_1+(1-\lambda) \beta_2}} \leq  \|T\|_{C_b^{\alpha_1} \to \mc{C}_b^{\beta_1}}^{\lambda} \|T\|_{C^{\alpha_2}_b \to \mc{C}^{\beta_2}_b}^{1-\lambda}
	  \label{def-eq5}
\end{equation}
for any $\alpha_i \geq 0$, $\beta_i \geq 0$ and $\lambda \in (0,1)$ where \begin{equation*}
	\|T\|_{X \to Y} := \inf\{c>0; \forall f \in Y, \|f\|_Y \leq 1: \, \, \|Tf\|_X \leq c\};
\end{equation*}
this inequality follows from the interpolation theorem, see e.\,g. \cite[Section 1.3.3]{triebel78} or \cite[Theorem 1.6]{lunardi}, and the fact that $\mc{C}_b^{\lambda \alpha}(\mbb{R}^d)$ is the real interpolation space $(C_b(\mbb{R}^d), C_b^{\alpha}(\mbb{R}^d)_{\lambda,\infty}$, cf.\ \cite[Theorem 2.7.2.1]{triebel78}.

\emph{L\'evy processes:} Throughout, $(\Omega,\mc{A},\mbb{P})$ is a probability space. A stochastic process $X_t: \Omega \to \mbb{R}^d$ is a ($d$-dimensional) L\'evy process if $X_0=0$ almost surely, $(X_t)_{t \geq 0}$ has independent and stationary increments and $t \mapsto X_t(\omega)$ is right-continuous with finite left-hand limits for almost all $\omega \in \Omega$. By the L\'evy--Khintchine formula, any L\'evy process is uniquely determined in distribution by its characteristic exponent $\psi: \mbb{R}^d \to \mbb{C}$ through the relation \begin{equation*}
	\mbb{E}\exp(i \xi X_t) = \exp(-t \psi(\xi)), \qquad t \geq 0, \, \xi \in \mbb{R}^d.
\end{equation*}
The characteristic exponent $\psi$ has the L\'evy--Khintchine representation \begin{equation*}
	\psi(\xi) = -ib \cdot \xi + \frac{1}{2} \xi \cdot Q \xi + \int_{y \neq 0} \left( 1-e^{iy \cdot \xi} + iy \cdot \xi \I_{(0,1)}(|y|) \right) \, \nu(dy), \quad \xi \in \mbb{R}^d,
\end{equation*}
where $(b,Q,\nu)$ is the L\'evy triplet consisting of a vector $b \in \mbb{R}^d$ (drift vector), a symmetric positive semi-definite matrix $Q \in \mbb{R}^{d \times d}$ (diffusion matrix) and a measure $\nu$ on $\mbb{R}^d \backslash \{0\}$ which satisfies the integrability condition $\int_{y \neq 0} \min\{1,|y|^2\} \, \nu(dy)<\infty$,  the so-called L\'evy measure. If the characteristic exponent $\psi$ of a L\'evy process $(X_t)_{t \geq 0}$ satisfies the Hartman--Wintner condition  \begin{equation*}
	\lim_{|\xi| \to \infty} \frac{\re \psi(\xi)}{\log(|\xi|)} = \infty,
\end{equation*}
then $X_t$ has a density $p_t$ with respect to Lebesgue measure for any $t>0$ and $p_t$ has bounded derivatives of arbitrary order; we refer to \cite{knop13} for a detailed discussion. \par
It follows from the independence and stationarity of the increments that any L\'evy process is a time-homogeneous Markov process, i.\,e.\ $P_t f(x) := \mbb{E}f(x+X_t)$ defines a Markov semigroup. We denote by $(A,\mc{D}(A))$ the infinitesimal generator associated with $(X_t)_{t \geq 0}$, \begin{align*}
	\mc{D}(A) &:= \left\{f \in C_{\infty}(\mbb{R}^d); \exists g \in C_{\infty}(\mbb{R}^d): \lim_{t \to 0} \left\| \frac{P_t f-f}{t} -g  \right\|_{\infty} = 0 \right\}, \\
	Af &:= \lim_{t \to 0} \frac{P_t f-f}{t}, \quad f \in \mc{D}(A).
\end{align*}
It is well-known that $C_c^{\infty}(\mbb{R}^d)$ is contained in $\mc{D}(A)$ and that \begin{equation*}
	Af(x) = b \cdot \nabla f(x) + \frac{1}{2} \tr(Q \cdot \nabla^2 f(x)) + \int_{y \neq 0} \left( f(x+y)-f(x)-\nabla f(x) \cdot y \I_{(0,1)}(|y|) \right) \, \nu(dy)
\end{equation*}
for any $f \in C_c^{\infty}(\mbb{R}^d)$, see e.\,g.\ \cite[Theorem 31.5]{sato}; here $(b,Q,\nu)$ denotes the L\'evy triplet of $(X_t)_{t \geq 0}$. Moreover,  the resolvent \begin{equation*}
	R_{\lambda} f(x) := \int_{(0,\infty)} e^{-\lambda t} P_t f(x) \, dt, \qquad f \in \mc{B}_b(\mbb{R}^d), \, \lambda>0, \, x \in \mbb{R}^d,
\end{equation*}
satisfies $R_{\lambda}(C_{\infty}(\mbb{R}^d)) = \mc{D}(A)$ for any $\lambda>0$. Our standard reference for L\'evy processes is the monograph \cite{sato} by Sato.

\section{Proof of Theorem~\ref{levy-9}} \label{proof}

The first two results in this section prepare the proof of Theorem~\ref{levy-9} but are of independent interest. 

\begin{prop} \label{levy-3}
	Let $(X_t)_{t \geq 0}$ be a L\'evy process with resolvent $(R_{\lambda})_{\lambda>0}$ and infinitesimal generator $(A,\mc{D}(A))$. Assume that the Hartman--Wintner condition \eqref{hw} holds. If the transition density $p_t$ satisfies \begin{equation}
		\int_{\mbb{R}^d} |\partial_{x_j} p_t(x)| \, dx \leq M e^{mt} t^{-1/\alpha}, \qquad t>0, \, j \in \{1,\ldots,d\}, \label{levy-eq11}
	\end{equation}
	for some constants $M>0$, $m \geq 0$, and $\alpha \in (0,2]$, then each of the following statements hold true. \begin{enumerate}
		\item\label{levy-3-i} $R_{\lambda}(\mc{B}_b(\mbb{R}^d)) \subseteq \mc{C}_b^{\alpha}(\mbb{R}^d)$ for any $\lambda \geq 3m$ and \begin{equation}
			\|R_{\lambda} f\|_{\mc{C}_b^{\alpha}(\mbb{R}^d)} \leq  K \|f\|_{\infty}, \qquad f \in \mc{B}_b(\mbb{R}^d) \label{levy-eq13}
		\end{equation}
		for a constant $K=K(m,\alpha,d,\lambda,M)$. 
		\item\label{levy-3-ii} If $\alpha>1$ then $R_{\lambda}(C_{\infty}(\mbb{R}^d)) \subseteq C_{\infty}^1(\mbb{R}^d)$ for any $\lambda >m$.
		\item\label{levy-3-iii} $\mc{D}(A) \subseteq \mc{C}^{\alpha}_{\infty}(\mbb{R}^d)$.
	\end{enumerate}
\end{prop}

\begin{proof}[Proof of Proposition~\ref{levy-3}]
	\firstpara{\eqref{levy-3-i}} It was shown in \cite[Lemma 4.1]{euler-maruyama} that \eqref{levy-eq11} implies \begin{equation}
		\int_{\mbb{R}^d} |\partial_{x_i} \partial_{x_j} p_{2t}(x)| \, dx \leq c(t)^2 \fa t>0,\, i,j=1,\ldots,d \label{levy-eq19}
	\end{equation}
	where $c(t) := M e^{mt} t^{-1/\alpha}$. For the readers' convenience we briefly explain the idea of the proof. By the Chapman--Kolmogorov equation, we have \begin{equation*}
		p_{2t}(x) = \int_{\mbb{R}^d} p_t(x-y) p_t(y) \, dy,
	\end{equation*}
	and so \begin{equation}
		\partial_{x_i} p_{2t}(x) = \int_{\mbb{R}^d} p_t(y) \partial_{x_i} p_t(x-y)  \, dy = \int_{\mbb{R}^d} p_t(x-z) \partial_{x_i} p_t(z) \, dz \label{levy-eq21}
	\end{equation}
	which implies \begin{equation*}
		\partial_{x_j} \partial_{x_i} p_{2t}(x) = \int_{\mbb{R}^d} \big( \partial_{x_j} p_t(x-z) \big) \big( \partial_{x_i} p_t(z) \big) \, dz.
	\end{equation*}
	Applying Tonelli's theorem we conclude that \begin{equation*}
		\int_{\mbb{R}^d} |\partial_{x_i} \partial_{x_j} p_{2t}(x)| \, dx \leq \left( \int_{\mbb{R}^d} |\partial_{x_i} p_t(z)| \, dz \right) \left( \int_{\mbb{R}^d} |\partial_{x_j} p_t(z)| \, dz \right) \leq c(t)^2,
	\end{equation*}
	and this proves \eqref{levy-eq19}. Iterating the procedure, we get \begin{equation}
		\int_{\mbb{R}^d} |\partial_x^{\beta} p_t(x)| \, dx \leq c(t)^{|\beta|} \fa \beta \in \mbb{N}_0^d. \label{levy-eq20}
	\end{equation}
	Now fix $f \in \mc{B}_b(\mbb{R}^d)$, $\lambda \geq 3m$ and $x,h \in \mbb{R}^d$. Since \begin{equation*}
		R_{\lambda} f(z) = \int_{(0,\infty)} \int_{\mbb{R}^d} e^{-\lambda t} f(y) p_t(y-z) \, dy \, dt
	\end{equation*}
	we have \begin{align*}
		|R_{\lambda} f(x+3h)+3 R_{\lambda} f(x+h)-3R_{\lambda} f(x+2h)-R_{\lambda} f(x)|
		&\leq I_1 + I_2
	\end{align*}
	where \begin{align*}
		I_1 &:= \left| \int_{t \leq |h|^{\alpha}} \int_{\mbb{R}^d} e^{-\lambda t} f(y) (p_t(y-x-3h)+3p_t(y-x-h)-3p_t(y-x-2h)-p_t(y-x)) \, dy \, dt \right| \\
		I_2 &:= \left| \int_{t>|h|^{\alpha}} \int_{\mbb{R}^d} e^{-\lambda t} f(y) (p_t(y-x-3h)+3p_t(y-x-h)-3p_t(y-x-2h)-p_t(y-x)) \, dy \, dt \right|.
	\end{align*}
	Using $\int_{\mbb{R}^d} |p_t(z+y)| \, dy = 1$ it follows from the triangle inequality that \begin{equation*}
		I_1
		\leq 8 \|f\|_{\infty} \sup_{z \in \mbb{R}^d} \int_{\mbb{R}^d} |p_t(z+y)| \, dy \int_0^{|h|^{\alpha}} \, dt
		= 8|h|^{\alpha} \|f\|_{\infty}.
	\end{equation*}
	To estimate $I_2$ we note that, by the multivariate version of Taylor's theorem, \begin{equation*}
		|p_t(y-x-3h)+3p_t(y-x-h)-3p_t(y-x-2h)-p_t(y-x)| \leq C |h|^3 \sum_{i=1}^3 \sum_{|\beta|=3} \int_0^1 |\partial_x^{\beta} p_t(y-x-i rh)| \, dr
	\end{equation*}
	for an absolute constant $C>0$.  Applying Tonelli's theorem and using \eqref{levy-eq20} we get \begin{align*}
		I_2
		&\leq C d^3 |h|^3 \|f\|_{\infty} \max_{|\beta|=3} \sum_{i=1}^3 \int_{t>|h|^{\alpha}}e^{-\lambda t} \int_0^1 \int_{\mbb{R}^d} |\partial_{x}^{\beta} p_t(y-x-irh)| \, dy \, dr \, dt \\
		&\leq C d^3 M^3 |h|^3 \|f\|_{\infty} \int_{t>|h|^{\alpha}} e^{(3m-\lambda)t} t^{-3/\alpha} \, dt
	\end{align*}
	As $\lambda \geq 3m$ and $\alpha \in (0,2]$ this implies \begin{align*}
		I_2
		\leq C d^3 M^3 |h|^3 \|f\|_{\infty} \int_{t>|h|^{\alpha}} t^{-3/\alpha} \, dt
		= C d^3 M^3 \frac{\alpha}{3-\alpha} \|f\|_{\infty} |h|^{\alpha}.
	\end{align*}
	Consequently, we have shown that \begin{align*}
		|R_{\lambda} f(x+3h)+3 R_{\lambda} f(x+h)-3R_{\lambda} f(x+2h)-R_{\lambda} f(x)| \leq C' |h|^{\alpha}, \qquad x,h \in \mbb{R}^d,
	\end{align*}
	 and by \eqref{def-eq3} this proves \eqref{levy-eq13}.
	 
	\para{\eqref{levy-3-ii}} If $\alpha>1$ then a straight-forward application of the differentiation lemma for parametrized integrals, see e.\,g.\  \cite[Theorem 12.5]{mims} or \cite[Proposition A.1]{euler-maruyama}, shows that \begin{equation*}
		\partial_{x_j} R_{\lambda} f(x) = - \int_{(0,\infty)} \int_{\mbb{R}^d} e^{-\lambda t} f(y) \partial_{x_j} p_t(y-x) \, dy \, dt
	\end{equation*}
	for $f \in \mc{B}_b(\mbb{R}^d)$, $\lambda>m$ and $j \in \{1,\ldots,d\}$. Since this clearly implies that \begin{equation*}
		\partial_{x_j} R_{\lambda} f(x) = - \int_{(0,\infty)} \int_{\mbb{R}^d} e^{-\lambda t} f(x+y) \partial_{x_j} p_t(y) \, dy \, dt
	\end{equation*}
	we can apply the dominated convergence theorem to conclude that $\lim_{|x| \to \infty} |\partial_{x_j} R_{\lambda} f(x)|=0$ for any $f \in C_{\infty}(\mbb{R}^d)$.
	
	\para{\eqref{levy-3-iii}} Since $R_{\lambda}(C_{\infty}(\mbb{R}^d)) = \mc{D}(A)$ for any $\lambda>0$, the assertion is obvious from \eqref{levy-3-i} and \eqref{levy-3-ii}.
\end{proof}

\begin{bem} \label{levy-5} \begin{enumerate}[wide, labelwidth=!, labelindent=0pt] 
	\item\label{levy-5-i}If there are constants $M>0$, $T>0$ and $\alpha \in (0,2]$ such that \begin{equation}
		\int_{\mbb{R}^d} |\partial_{x_i} p_t(x)| \, dx \leq M t^{-1/\alpha}, \qquad t \in (0,T], \, i \in \{1,\ldots,d\}, \label{levy-eq15}
		\end{equation}
		then there exists $m \geq 0$ such that \eqref{levy-eq11} holds.  \emph{Indeed:} Fix $t \in (0,T)$ and $i \in \{1,\ldots,d\}$. It follows from \eqref{levy-eq21} that \begin{align*}
			\int_{\mbb{R}^d} |\partial_{x_i} p_{2t}(x)| \, dx 
			&\leq \int_{\mbb{R}^d} |\partial_{x_i} p_t(z)|  \left( \int_{\mbb{R}^d} |p_t(x-z)| \, dx \right) \, dz \\
			&= \int_{\mbb{R}^d} |\partial_{x_i} p_t(z)| \, dz 
			\leq c(t) := M t^{-1/\alpha},
		\end{align*}
		which gives \begin{equation*}
			\int_{\mbb{R}^d} |\partial_{x_i} p_s(x)| \, dx \leq c(s/2) = M 2^{1/\alpha} s^{-1/\alpha} \fa s  \in [T,2T).
		\end{equation*}
		By iteration we find that \begin{equation*}
			\int_{\mbb{R}^d} |\partial_{x_i} p_s(x)| \, dx \leq M (2^{1/\alpha})^k s^{-1/\alpha} \fa s \in [2^{k-1}T,2^kT),\, k \in \mbb{N}.
		\end{equation*}
		Hence, \eqref{levy-eq11} holds for $m := \log(2^{1/(\alpha T)})$. 
		\item\label{levy-5-ii} If \eqref{levy-eq15} holds for some $\alpha \geq 0$, then $\alpha \leq 2$. Indeed: The Fourier transform of $x \mapsto \partial_{x_j} p_t(x)$ equals $i \xi_j e^{-t \psi(\xi)}$, and therefore \begin{equation*}
			\sup_{\xi \in \mbb{R}^d} |\xi_j e^{-t \psi(\xi)}| \leq \|\partial_{x_j} p_t\|_{L^1} = \int_{\mbb{R}^d} |\partial_{x_j} p_t(x)| \, dx.
		\end{equation*}
		Since the characteristic exponent$\psi$ satisfies $|\psi(\xi)| \leq c(1+|\xi|^2)$, $\xi \in \mbb{R}^d$, for some constant $c>0$ this gives \begin{equation*}
			\int_{\mbb{R}^d} |\partial_{x_j} p_t(x)| \, dx \geq \sup_{\xi \in \mbb{R}^d} |\xi_j e^{-ct (1+|\xi|^2)}  \geq c' t^{-1/2}.
		\end{equation*}
		\end{enumerate}
 \end{bem}

In Proposition~\ref{levy-3} we have seen that $R_{\lambda} f \in \mc{C}_b^{\alpha}(\mbb{R}^d)$ for $f \in C_{b}(\mbb{R}^d)$.  Our next result, Corollary~\ref{levy-7}, shows that, more generally, \begin{equation*}
	f \in \mc{C}_b^{\beta}(\mbb{R}^d) \implies R_{\lambda} f \in \mc{C}_b^{\alpha+\beta}(\mbb{R}^d)
\end{equation*}
for any $\beta \geq 0$. 

\begin{kor} \label{levy-7}
	Let $(X_t)_{t \geq 0}$ be a L\'evy process with resolvent $(R_{\lambda})_{\lambda>0}$ and infinitesimal generator $(A,\mc{D}(A))$  such that its characteristic exponent $\psi$ satisfies the Hartman--Wintner condition \eqref{hw}. If there exist constants $M>0$, $m \geq 0$ and $\alpha \in (0,2]$ such that the transition density $p_t$ satisfies \begin{equation*}
		\int_{\mbb{R}^d} |\partial_{x_i} p_t(x)| \, dx \leq  M e^{mt} t^{-1/\alpha}, \qquad t>0, \, i \in \{1,\ldots,d\},
	\end{equation*}
	then there exists for any $k \in \mbb{N}$ a constant $K=K(d,M,\alpha,\lambda,k)>0$ such that \begin{equation}
		\|R_{\lambda} f\|_{\mc{C}_b^{\alpha+\beta}(\mbb{R}^d)} \leq K \|f\|_{\mc{C}_b^{\beta}(\mbb{R}^d)} \fa f \in \mc{C}_b^{\beta}(\mbb{R}^d), \, \beta \in (0,k), \, \lambda \geq 3m. \label{levy-eq25}
	\end{equation}
\end{kor}

\begin{proof}
	Fix $\lambda \geq 3m$ and $k \in \mbb{N}$, and let $f \in C_b^k(\mbb{R}^d)$. Since \begin{equation*}
		R_{\lambda} f(x) = \int_{(0,\infty)} \int_{\mbb{R}^d} e^{-\lambda t} f(x+y) p_t(y) \, dy \, dt, \qquad x \in \mbb{R}^d
	\end{equation*}
	it follows from an application of the differentiation lemma for parametrized integrals that \begin{equation}
		\partial_x^{\gamma} R_{\lambda} f(x) = \int_{(0,\infty)} \int_{\mbb{R}^d} e^{-\lambda t} \partial^{\gamma}_x f(x+y) p_t(y) \, dy \, dt = R_{\lambda}(\partial_x^{\gamma} f)(x) \label{levy-eq27}
	\end{equation}
	for any multi-index $\gamma \in \mbb{N}_0^d$ with $|\gamma| := \sum_{i=1}^d \gamma_i \leq k$. By Proposition~\ref{levy-3}, there exists a constant $K>0$ such that \begin{equation*}
		\|\partial^{\gamma} R_{\lambda} f\|_{\mc{C}_b^{\alpha}(\mbb{R}^d)} 
		= \|R_{\lambda}(\partial^{\gamma} f)(x)\|_{\mc{C}_b^{\alpha}(\mbb{R}^d)} 
		\leq K \|\partial^{\gamma} f\|_{\infty},
	\end{equation*}
	and so, by \eqref{def-eq3}, \begin{equation*}
		\|R_{\lambda} f\|_{\mc{C}_b^{\alpha+k}(\mbb{R}^d)} \leq c K \|f\|_{C_b^k(\mbb{R}^d)} \fa f \in C_b^k(\mbb{R}^d)
	\end{equation*}
	for some constant $c=c(k) \geq 1$. On the other hand, Proposition~\ref{levy-3} shows that \begin{equation*}
		\|R_{\lambda} h\|_{\mc{C}_b^{\alpha}(\mbb{R}^d)} \leq K \|h\|_{C_b(\mbb{R}^d)} \fa h \in C_b(\mbb{R}^d).
	\end{equation*}
	Applying the interpolation theorem, cf.\ \eqref{def-eq5}, we thus find that  \begin{equation*}
		\|R_{\lambda} h\|_{\mc{C}_b^{\alpha+\beta}(\mbb{R}^d)} \leq c K \|h\|_{\mc{C}_b^{\beta}(\mbb{R}^d)}, \qquad h \in \mc{C}_b^{\beta}(\mbb{R}^d)
	\end{equation*}
	for any $\beta \in (0,k)$. 
\end{proof}

A close look at the proof of Corollary~\ref{levy-7} shows that $R_{\lambda} f \in C_{\infty}^{\lsem \beta \rsem + \lsem \alpha \rsem}(\mbb{R}^d)$ for any $f \in C_{\infty}^{\beta}(\mbb{R}^d)$, cf.\ \eqref{def-eq1} for the definition of $\lsem \alpha \rsem$ and $\lsem \beta \rsem$; this is a consequence of \eqref{levy-eq27} and Proposition~\ref{levy-3}\eqref{levy-3-ii}. \par \medskip


We are now ready to prove Theorem~\ref{levy-9}.

\begin{proof}[Proof of Theorem~\ref{levy-9}]
	Let $f \in \mc{D}(A)$ be such that $Af+ \varrho f =g\in \mc{C}_b^{\delta}(\mbb{R}^d)$ for some $\delta \geq 0$ and $\varrho \in \mbb{R}$. It follows from \eqref{levy-eq1} and Remark~\ref{levy-5}\eqref{levy-5-i} that $p_t$ satisfies \eqref{levy-eq11} for some $m \geq 0$, and we set $\lambda:=3m+1$. Since $\mc{D}(A) = R_{\lambda}(C_{\infty}(\mbb{R}^d))$ there exists $h \in C_{\infty}(\mbb{R}^d)$ such that $f = R_{\lambda} h$. As $Af+\varrho f=g$ we have $AR_{\lambda} h + \varrho R_{\lambda} h = g$, and using that $(\lambda \id-A) R_{\lambda} h= h$ this gives \begin{equation}
		h = (\lambda+\varrho) R_{\lambda} h - g. \label{levy-eq35}
	\end{equation}
	We claim that for any $k \in \mbb{N}_0$ there exists a constant $c_k>0$ (not depending on $f$, $g$) such that \begin{equation}
		h \in \mc{C}^{\delta \wedge (k \alpha)}_b(\mbb{R}^d) \quad \text{and} \quad \|h\|_{\mc{C}_b^{\delta \wedge (k \alpha)}(\mbb{R}^d)} \leq c_k \left(\|g\|_{\mc{C}^{\delta}_b(\mbb{R}^d)} + \|h\|_{\infty} \right); \label{levy-eq37} 
	\end{equation}
	we prove \eqref{levy-eq37} by induction. For $k=0$ the assertion  is obvious as $h \in C_b(\mbb{R}^d)$ and $\|\lambda R_{\lambda} h\|_{\infty} \leq \|h\|_{\infty}$. Now suppose that \eqref{levy-eq37} holds for some $k \in \mbb{N}_0$.  It follows from Corollary~\ref{levy-7} (with $\beta=\delta \wedge (\alpha k)$) that $R_{\lambda} h \in \mc{C}^{\alpha + (\delta \wedge (\alpha k))}_b(\mbb{R}^d)$ and \begin{equation*}
		\|R_{\lambda} h\|_{\mc{C}_b^{\alpha+(\delta \wedge (\alpha k))}(\mbb{R}^d)} \leq K \|h\|_{\mc{C}_b^{\delta \wedge (k \alpha)}(\mbb{R}^d)}.
	\end{equation*}
	Since, by assumption, $g \in \mc{C}^{\delta}_b(\mbb{R}^d)$ we find from \eqref{levy-eq35} that $h \in \mc{C}^{((k+1) \alpha) \wedge \delta}_b(\mbb{R}^d)$ and \begin{align*}
		\|h\|_{\mc{C}_b^{\delta \wedge ((k+1)\alpha)}(\mbb{R}^d)}  
		&\leq \|g\|_{\mc{C}_b^{\delta}(\mbb{R}^d)} + |\lambda+\varrho| \, \|R_{\lambda} h\|_{\mc{C}_b^{\alpha+(\delta \wedge (\alpha k))}(\mbb{R}^d)} \\
		&\leq \|g\|_{\mc{C}_b^{\delta}(\mbb{R}^d)} + K |\lambda+\varrho| \, \| h\|_{\mc{C}_b^{\delta \wedge (k \alpha)}(\mbb{R}^d)} \\
		&\leq \|g\|_{\mc{C}_b^{\delta}(\mbb{R}^d)} + K |\lambda+\varrho| c_{k} \left( \|g\|_{\mc{C}_b^{\delta}(\mbb{R}^d)} + \|h\|_{\infty} \right),
	\end{align*}
	i.e. \eqref{levy-eq37} holds for $k+1$. We conclude that \eqref{levy-eq37} holds for any $k \in \mbb{N}_0$. If we choose $k \in \mbb{N}$ sufficiently large  such that $k \alpha \geq \delta$, then we find in particular $h \in \mc{C}^{\delta}_b(\mbb{R}^d)$. Applying once more Corollary~\ref{levy-7} we obtain that \begin{equation*}
		\|R_{\lambda} h\|_{\mc{C}_b^{\alpha+\delta}(\mbb{R}^d)} 
		\leq K \|h\|_{\mc{C}_b^{\delta}(\mbb{R}^d)}
		\leq K c_k \big( \|g\|_{\mc{C}_b^{\delta}(\mbb{R}^d)} + \|h\|_{\infty} \big).
	\end{equation*}
	Finally we note that $h=(\lambda \id-A)f$ implies \begin{equation*}
		\|h\|_{\infty} \leq \lambda \|f\|_{\infty} + \|Af\|_{\infty} = \lambda \|f\|_{\infty} + \|g\|_{\infty},
	\end{equation*}
	and therefore we conclude that \begin{equation*}
		\|f\|_{\mc{C}_b^{\alpha+\delta}(\mbb{R}^d)} 
		= \|R_{\lambda} h\|_{\mc{C}_b^{\alpha+\delta}(\mbb{R}^d)} 
		\leq L \big( \|g\|_{\mc{C}_b^{\delta}(\mbb{R}^d)} + \|f\|_{\infty} \big). \qedhere
	\end{equation*}
\end{proof}

\section{Examples} \label{ex}

In this section we illustrate Theorem~\ref{levy-9} with some examples and applications.  \par \medskip

Applying Theorem~\ref{levy-9} to L\'evy processes with continuous sample paths, we recover a classical result, see e.\,g.\ \cite{gilbarg},  on the regularity of the solutions to the second order elliptic differential equation \begin{equation*}
	\varrho f(x) + \sum_{j=1}^d b_j \partial_{x_j} f(x)  + \frac{1}{2} \sum_{i=1}^d \sum_{j=1}^d q_{ij} \partial_{x_i} \partial_{x_j} f(x) = g(x).
\end{equation*}

\begin{bsp} \label{ex-0}
	Let $(B_t)_{t \geq 0}$ be a $d$-dimensional Brownian motion, $b \in \mbb{R}^d$, and let $Q \in \mbb{R}^{d \times d}$ be a symmetric positive definite matrix. The infinitesimal generator $(A,\mc{D}(A))$ of the L\'evy process $X_t := bt + Q \cdot B_t$ satisfies \begin{equation*}
		Af(x) = b \cdot \nabla f(x) + \frac{1}{2} \tr(Q \cdot \nabla^2 f(x)), \qquad f \in C_{\infty}^2(\mbb{R}^d),
	\end{equation*}
	and has the following properties: \begin{enumerate}
		\item $\mc{D}(A) \subseteq \mc{C}_{\infty}^2(\mbb{R}^d)$,
		\item If $Af+\varrho f=g \in \mc{C}_b^{\delta}(\mbb{R}^d)$ for some $\delta \geq 0$ and $\varrho \in \mbb{R}$, then $f \in \mc{C}_b^{2+\delta}(\mbb{R}^d)$. Moreover, there exists a finite constant $L=L(d,\delta,\varrho)>0$ such that \begin{equation*}
			\|f\|_{\mc{C}_b^{\delta+2}(\mbb{R}^d)} \leq L(\|Af\|_{\mc{C}_b^{\delta}(\mbb{R}^d)} + \|f\|_{\infty}), \qquad f \in \mc{D}(A).
		\end{equation*}
	\end{enumerate}
\end{bsp}

For the definition of the H\"older spaces $\mc{C}_{\infty}^{\alpha}(\mbb{R}^d)$ and $\mc{C}_b^{\alpha}(\mathbb{R}^d)$ we refer the reader to Section~\ref{def}. Since there is a closed formula for the transition density $p_t$ of $X_t$, it can be easily verified that the assumptions of Theorem~\ref{levy-9} are satisfied for $\alpha=2$, and this proves the assertion of Example~\ref{ex-0}. \par \medskip

Our next result applies to a large class of L\'evy processes, including stable L\'evy processes. 

\begin{bsp} \label{ex-1}
	Let $(L_t)_{t \geq 0}$ be a pure-jump L\'evy process with infinitesimal generator $(A,\mc{D}(A))$. Assume that its L\'evy measure $\nu$ satisfies \begin{equation*}
		\nu(A) \geq \int_0^{r_0} \!\! \int_{\mbb{S}^{d-1}} \I_A(r \theta) r^{-1-\alpha} \, \mu(d\theta) \, dr + \int_{r_0}^{\infty}\!\! \int_{\mbb{S}^{d-1}} \I_A(r\theta) r^{-1-\beta} \, \mu(d\theta) \, dr, \quad A \in \mc{B}(\mbb{R}^d \backslash \{0\})
	\end{equation*}
	for some constants $\alpha \in (0,2)$, $\beta \in (0,\infty]$ and a finite measure $\mu$ on the unit sphere $\mbb{S}^{d-1} \subseteq \mbb{R}^d$ which is non-degenerate, in the sense that its support is not contained in $\mbb{S}^{d-1} \cap V$ where $V \subseteq \mbb{R}^d$ is a lower-dimensional subspace. Then: \begin{enumerate}
		\item $\mc{D}(A) \subseteq \mc{C}_{\infty}^{\alpha}(\mbb{R}^d)$,
		\item If $f \in \mc{D}(A)$ is such that $Af+\varrho f=g \in \mc{C}_b^{\delta}(\mbb{R}^d)$ for some $\delta \geq 0$ and $\varrho \in \mbb{R}$, then $f \in \mc{C}_b^{\alpha+\delta}(\mbb{R}^d)$. Moreover, there exists for any $\delta \geq 0$ a finite constant $L=L(\alpha,\beta,\mu,d,\delta,\varrho)$ such that \begin{equation*}
			\|f\|_{\mc{C}_{b}^{\alpha+\delta}(\mbb{R}^d)} \leq L (\|Af\|_{\mc{C}_{b}^{\delta}(\mbb{R}^d)}+\|f\|_{\infty}), \qquad f \in \mc{D}(A).
		\end{equation*}
	\end{enumerate}
\end{bsp}

Example~\ref{ex-1} is a direct consequence of Theorem~\ref{levy-9}, Remark~\ref{levy-1}\eqref{levy-1-ii} and \cite[Example 1.5]{ssw12}.    \par \medskip

The remaining part of this section is devoted to L\'evy processes whose characteristic exponent $\psi$ satisfies \begin{equation*}
	\re \psi(\xi) \asymp |\xi|^{\alpha} \quad \text{as $|\xi| \to \infty$}
\end{equation*}
for some $\alpha \in (0,2)$. This class of L\'evy processes covers many important and interesting examples, e.\,g.\  \begin{itemize}
	\item isotropic stable, relativistic stable and tempered stable L\'evy processes, 
	\item subordinated Brownian motions with characteristic exponent $\psi(\xi) = f(|\xi|^2)$ for a Bernstein function $f$ satisfying $f(\lambda) \asymp \lambda^{\alpha/2}$ for large $\lambda$, cf.\ \cite{bernstein} for details.
	\item  L\'evy processes with symbol of the form \begin{equation*}
	\psi(\xi) = |\xi|^{\alpha} + |\xi|^{\beta}, \qquad \xi \in \mbb{R}^d,
\end{equation*}
for $\beta \in (0,\alpha)$. \end{itemize}

\begin{thm} \label{levy-11}
	Let $(X_t)_{t \geq 0}$ be a L\'evy process with infinitesimal generator $(A,\mc{D}(A))$. If the characteristic exponent $\psi$ satisfies the sector condition, $|\im \psi(\xi)| \leq c \re \psi(\xi)$, and \begin{equation}
	\re \psi(\xi) \asymp |\xi|^{\alpha} \quad \text{as $|\xi| \to \infty$} \label{levy-eq45}
	\end{equation}
	for some $\alpha \in (0,2)$, then: \begin{enumerate}
		\item\label{levy-11-i} $\mc{C}_{\infty}^{\alpha+}(\mbb{R}^d):= \bigcup_{\beta>\alpha} \mc{C}_{\infty}^{\beta}(\mbb{R}^d) \subseteq \mc{D}(A) \subseteq \mc{C}_{\infty}^{\alpha}(\mbb{R}^d)$.
		\item\label{levy-11-ii} If $f \in \mc{D}(A)$ is such that $Af+\varrho f \in \mc{C}_b^{\delta}(\mbb{R}^d)$ for some $\delta \geq 0$ and $\varrho \in \mbb{R}$, then $f \in \mc{C}^{\alpha+\delta}_b(\mbb{R}^d)$ and
		\begin{equation*}
					\|f\|_{\mc{C}_{b}^{\alpha+\delta}(\mbb{R}^d)} \leq L (\|Af\|_{\mc{C}_{b}^{\delta}(\mbb{R}^d)}+\|f\|_{\infty}), \qquad f \in \mc{D}(A).
				\end{equation*}
			for some constant $L=L(\delta,\varrho,\alpha)$. 
		\item\label{levy-11-iii} $\mc{D}(A)$ is an algebra, i.\,e.\ $f,g \in \mc{D}(A)$ implies $f \cdot g \in \mc{D}(A)$, and \begin{equation}
		A(f \cdot g)(x) = f(x) Ag(x)+ g(x) Af(x) + \Gamma(f,g)(x), \qquad f,g \in \mc{D}(A) \label{levy-eq46}
		\end{equation}
		where \begin{equation}
		\Gamma(f,g)(x) := \int_{y \neq 0} \big( f(x+y)-f(x) \big) \big( g(x+y)-g(x) \big) \, \nu(dy), \qquad x \in \mbb{R}^d \label{levy-eq48}
		\end{equation}
		is the Carr\'e du Champ operator, cf.\ Remark~\ref{levy-12-iii}; here $\nu$ denotes the L\'evy measure of $(X_t)_{t \geq 0}$.
	\end{enumerate}
\end{thm}

\begin{bem} \label{levy-12}  \begin{enumerate}[wide, labelwidth=!, labelindent=0pt]
		\item\label{levy-12-i} The proof of Theorem~\ref{levy-11}\eqref{levy-11-iii} shows the following slightly more general statement: Let $(X_t)_{t \geq 0}$ be a L\'evy process with generator $(A,\mc{D}(A))$ and characteristic exponent $\psi$ satisfying \begin{equation*} 
		\limsup_{|\xi| \to \infty} \frac{|\psi(\xi)|}{|\xi|^{\alpha}} < \infty
		\end{equation*}
		for some $\alpha \in (0,2)$. Let $f,g \in \mc{D}(A)$ be such that \begin{equation*}
		|f(x+y)-f(x)| \leq C_1 |y|^{\beta_1} \quad \text{and} \quad |g(x+y)-g(x)| \leq C_2 |y|^{\beta_2} 
		\end{equation*}
		for all $x \in \mbb{R}^d$ and $|y| \leq 1$. If $\beta_1+\beta_2>\alpha$ then $f \cdot g \in \mc{D}(A)$ and \eqref{levy-eq46} holds. 
		\item\label{levy-12-ii} Theorem~\ref{levy-11} can be used to establish inclusions of the form $\mc{D}(A) \subseteq \mc{D}(L)$ for L\'evy generators $A$ and $L$. More precisely, if $(X_t)_{t \geq 0}$ and $(Y_t)_{t \geq 0}$ are L\'evy processes with characteristic exponent $\psi$ and $\tilde{\psi}$, respectively, which both satisfy the sector condition and \begin{equation*}
		\re \psi(\xi) \asymp |\xi|^{\alpha} \quad \text{and} \quad \re \tilde{\psi}(\xi) \asymp |\xi|^{\beta} \quad \text{as $|\xi| \to \infty$}
		\end{equation*}
		for $\alpha<\beta$, then Theorem~\ref{levy-11} shows that the domain of the generator of $(X_t)_{t \geq 0}$ is contained in the domain of the generator of $(Y_t)_{t \geq 0}$. For instance, the domain $\mc{D}(A^{(\alpha)})$ of the infinitesimal generator associated with the isotropic $\alpha$-stable L\'evy process, $\alpha \in (0,2]$, satisfies $\mc{D}(A^{(\beta)}) \subseteq \mc{D}(A^{(\alpha)})$ for $\alpha<\beta$; this is a well-known result which can be, for instance, also proved using subordination, cf.\ \cite[Theorem 13.6]{bernstein}. 
		\item\label{levy-12-iii} In contrast to other authors, we consider the Carr\'e du champ operator $\Gamma$ as an operator on $C_{\infty}(\mbb{R}^d)$ and not on $L^2(dx)$. For further information on the Carr\'e du champ operator we refer the reader to \cite{bouleau91,meyer4}. 
	\end{enumerate}
\end{bem}

\begin{proof}[Proof of Theorem~\ref{levy-11}]
	Under the growth condition \eqref{levy-eq45} it is shown in \cite{ssw12} that the semigroup $P_t f(x) := \mbb{E}f(x+X_t)$ satisfies the gradient estimate \begin{equation*}
	\|\nabla P_t f\|_{\infty} \leq ct^{-1/\alpha} \|f\|_{\infty}, \qquad t \in (0,1], \, f \in \mc{B}_b(\mbb{R}^d)
	\end{equation*}
	for some absolute constant $c>0$. Since this implies $\int |\nabla p_t(x)| \, dx \leq c' t^{-1/\alpha}$, cf.\ Remark~\ref{levy-1}\eqref{levy-1-ii}, Theorem~\ref{levy-9} gives \eqref{levy-11-ii} and $\mc{D}(A) \subseteq \mc{C}_{\infty}^{\alpha}(\mbb{R}^d)$ . To prove $\mc{C}_{\infty}^{\beta}(\mbb{R}^d) \subseteq \mc{D}(A)$, $\beta>\alpha$, we need some properties of the L\'evy triplet $(b,Q,\nu)$ which are consequences of the growth condition \eqref{levy-eq45} and the sector condition. As $\alpha<2$ it follows from \cite[Lemma A.3]{ihke} that $Q=0$ and \cite[Lemma A.3]{ihke} also shows $b = \int_{|y|<1} y \, \nu(dy)$ if $\alpha<1$. Moreover,  \begin{equation}
	\int_{0<|y|<1} |y|^{\beta} \, \nu(dy)<\infty \fa \beta>\alpha, \label{levy-eq47}
	\end{equation}
	see e.\,g.\ \cite{blumen61,rs-growth} or \cite[Lemma A.2]{ihke} for a detailed proof.  By \cite[Theorem 4.1]{ihke}, these properties of the L\'evy triplet imply that $\mc{C}_{\infty}^{\beta}(\mbb{R}^d)\subseteq \mc{D}(A)$ for $\beta>\alpha$. It remains to prove \eqref{levy-11-iii}. Let $f,g \in \mc{D}(A)$ and fix $x \in \mbb{R}^d$.  We will first show that \begin{equation}
	\lim_{t \to 0} \frac{\mbb{E}([f(x+X_t)-f(x)] \cdot [g(x+X_t)-g(x)])}{t}
	=  \Gamma(f,g)(x) \label{levy-eq49}
	\end{equation}
	with $\Gamma(f,g)(x)$ defined in \eqref{levy-eq48}. Pick a truncation function $\chi \in C_c^{\infty}(\mbb{R}^d)$, $\I_{B(0,1)} \leq \chi \leq \I_{B(0,2)}$ and set $\chi_{\eps}(y) := \chi(\eps^{-1} y)$.  Since the function $y \mapsto (1-\chi_{\eps}(y)) (f(x+y)-f(x)) (g(x+y)-g(x))$ is continuous and equal to zero in a neighbourhood of $x$, the weak convergence $t^{-1} \mbb{P}(X_t \in \cdot) \to \nu(\cdot)$ as $t \to 0$, cf.\ \cite[Corollary 8.9]{sato} or \cite[Corollary 3.3]{ihke}, yields \begin{align*}
	&\lim_{t \to 0} \frac{\mbb{E} \big( (1-\chi_{\eps}(X_t)) (f(x+X_t)-f(x)) (g(x+X_t)-g(x)) \big) }{t} \\
	&\quad= \int_{y \neq 0} (1-\chi_{\eps}(y)) (f(x+y)-f(x)) (g(x+y)-g(x)) \, \nu(dy).
	\end{align*}
	By \eqref{levy-11-i}, we have $f,g \in \mc{D}(A) \subseteq \mc{C}_{\infty}^{\alpha}(\mbb{R}^d)$ and so \begin{equation}
	\left| \left( f(x+y)-f(x) \right) \left( g(x+y)-g(x) \right) \right| \leq C \min\{|y|^{2(\alpha \wedge 1)},1\}, \qquad x,y \in \mbb{R}^d; \label{levy-st2} \tag{$\star$}
	\end{equation}
	using \eqref{levy-eq47} a straight-forward application of the dominated convergence theorem now shows that the right-hand side of the previous equation converges to $\Gamma(f,g)(x)$ as $\eps \to 0$.  On the other hand, $\spt \chi_{\eps} \subseteq B(0,2\eps)$ and \eqref{levy-st2} give  \begin{align*}
	\left| \mbb{E}\big( \chi_{\eps}(X_t) (f(x+X_t)-f(x)) (g(x+X_t)-g(x)) \big) \right|
	&\leq \mbb{E}(|X_t|^{2(\alpha \wedge 1)} 1_{\{|X_t| \leq 2 \eps\}}) \\
	&= \int_{(0,2\eps)} \mbb{P}(|X_t|^{2(\alpha \wedge 1)} \geq r) \, dr.
	\end{align*}
	Applying the maximal inequality, see e.\,g.\ \cite[Corollary 5.2]{ltp}, and invoking the growth condition \eqref{levy-eq45} we thus find \begin{align*}
	\left| \mbb{E}\big( \chi_{\eps}(X_t) (f(X_t)-f(x)) (g(X_t)-g(x)) \big) \right|
	&\leq ct \int_0^{2 \eps} \sup_{|\xi| \leq r^{-1/2(\alpha \wedge 1)}} |\psi(\xi)| \, dr \\
	&\leq c' t \int_0^{2\eps} r^{-\alpha/2(\alpha \wedge 1)} \, dr
	\end{align*}
	for absolute constants $c,c'>0$. As $\int_0^1 r^{-\alpha/2(\alpha \wedge 1)} \, dr < \infty$ an application of the monotone convergence theorem yields \begin{equation*}
	\limsup_{\eps \to 0} \limsup_{t \to 0} \frac{\left| \mbb{E}\big( \chi_{\eps}(X_t) (f(X_t)-f(x)) (g(X_t)-g(x)) \big) \right|}{t} = 0;
	\end{equation*}
	combining this with the earlier consideration, this proves \eqref{levy-eq49}. Now let $f,g \in \mc{D}(A)$ and fix $x \in \mbb{R}^d$. Clearly, 
	\begin{align*}
	\mbb{E}\left(  f(x+X_t) g(x+X_t) \right)-f(x) g(x)
	&= f(x) \mbb{E}(g(x+X_t)-g(x)) + g(x) \mbb{E}(f(x+X_t)-f(x)) \\
	&\quad+ \mbb{E} \left( (f(x+X_t)-f(x))(g(x+X_t)-g(x)) \right).
	\end{align*}
	Dividing both sides by $t$ and letting $t$ to $0$ we obtain from \eqref{levy-eq49} and the very definition of the generator $A$ that \begin{equation*}
	L(f \cdot g)(x):=\lim_{t \to 0} \frac{1}{t} \left[\mbb{E} \left(  f(x+X_t) g(x+X_t) \right)-f(x) g(x) \right] = f(x) Ag(x) + g(x) Af(x)+\Gamma(f,g)(x).
	\end{equation*}
	Using the estimate \eqref{levy-st2} it follows from the dominated convergence theorem that $\Gamma(f,g) \in C_{\infty}(\mbb{R}^d)$, and, hence, $L(f \cdot g) \in C_{\infty}(\mbb{R}^d)$. This implies $f \cdot g \in \mc{D}(A)$ and $A(f \cdot g)=L(f \cdot g)$, see e.\,g.\ \cite[Theorem 1.33]{ltp}.
\end{proof}

\section{Domain of the infinitesimal generator of two-dimensional Cauchy process}  \label{cauchy}

Let $(X_t)_{t \geq 0}$ be an isotropic $\alpha$-stable L\'evy process, $\alpha \in (0,2)$. It follows from Theorem~\ref{levy-11}\eqref{levy-11-i} that the domain $\mc{D}(A)$ of the infinitesimal generator of $(X_t)_{t \geq 0}$ satisfies \begin{equation*}
	\mc{C}_{\infty}^{\alpha+}(\mbb{R}^d) := \bigcup_{\beta>\alpha} \mc{C}_{\infty}^{\beta}(\mbb{R}^d) \subseteq \mc{D}(A) \subseteq \mc{C}_{\infty}^{\alpha}(\mbb{R}^d).
\end{equation*}
In this section we investigate whether this is the optimal way to describe $\mc{D}(A)$ in terms of H\"older spaces and whether the inclusions are strict.  The case $\alpha=1$ is particularly interesting since there are several functions spaces which are possible candidates to describe the domain: \begin{itemize}
	\item the space of Lipschitz continuous functions $\lip(\mbb{R}^d)$ vanishing at infinity,
	\item the space $C^1_{\infty}(\mbb{R}^d)$ of differentiable functions  vanishing at infinity,
	\item the Zygmund space $\mc{C}_{\infty}^1(\mbb{R}^d)$ of functions $f$ vanishing at infinity and satisfying \begin{equation*}
		|\Delta_h^2 f(x)| = |f(x+h)+f(x-h)-2f(x)| \leq C |h|, \qquad x,h \in \mbb{R}^d,
	\end{equation*}
	for some constant $C>0$, see \eqref{def-eq2}. 
\end{itemize}
We will show that the domain $\mc{D}(A)$ of the generator of the two-dimensional Cauchy process has the following properties: \begin{itemize}
	\item There exists a function $f \in C_{\infty}^1(\mbb{R}^2)$ which is not in $\mc{D}(A)$, cf.\ Proposition~\ref{levy-17}.
	\item There exists a function $f \in \mc{D}(A)$ which is not Lipschitz continuous, cf.\ Theorem~\ref{levy-19}.
\end{itemize}
This implies that \begin{equation*}
	\mc{D}(A) \not\subseteq \lip (\mbb{R}^2)\cap C_{\infty}(\mbb{R}^2) \quad \mc{D}(A) \not \subseteq C_{\infty}^1(\mbb{R}^2) \quad C_{\infty}^1(\mbb{R}^2) \not \subseteq \mc{D}(A) \quad \lip(\mbb{R}^2) \cap C_{\infty}(\mbb{R}^2) \not \subseteq \mc{D}(A)
\end{equation*}
which clearly shows that the function spaces $\lip(\mbb{R}^2) \cap C_{\infty}(\mbb{R}^2)$ and $C_{\infty}^1(\mbb{R}^2)$ are not well suited for describing $\mc{D}(A)$. We conclude that \begin{equation*}
	\mc{C}_{\infty}^{1+}(\mbb{R}^2) \subseteq \mc{D}(A) \subseteq \mc{C}_{\infty}^1(\mbb{R}^2)
\end{equation*}
is the best possible way to describe $\mc{D}(A)$ in terms of H\"older spaces and, moreover, the inclusions are strict.

\begin{prop} \label{levy-17}
	Let $(X_t)_{t \geq 0}$ be a $d$-dimensional Cauchy process with generator $(A,\mc{D}(A))$. Then there exists a function $f \in C_{\infty}^1(\mbb{R}^d)$ which is not in $\mc{D}(A)$.
\end{prop}

\begin{proof}
	Let $\chi \in C_c^{\infty}(\mbb{R}^d)$ be a cut-off function such that $\I_{B(0,1/4)} \leq \chi \leq \I_{B(0,1/2)}$, and define \begin{equation*}
	f(x) := \I_{\mbb{R}^d \backslash \{0\}}(x)\frac{|x|}{|\log |x||} \chi(x), \qquad x \in \mbb{R}^d. 
	\end{equation*}
	If we set  $\varphi(r) := |\log r|^{-1}$, then $\varphi(r) \to 0$ as $r \to 0$ and \begin{equation*}
		f(x) = 0 + 0 \cdot x + |x| \varphi(|x|), \qquad |x|< \frac{1}{2}
	\end{equation*}
	which shows that $f$ is differentiable at $x=0$ and $\nabla f(0)=0$. For $x \neq 0$ the differentiability is obvious. Clearly, $f$ and its derivatives are vanishing at infinity, and so $f \in C_{\infty}^1(\mbb{R}^d)$. 	Since the transition density $p_t$ of $X_t$ satisfies \begin{equation*}
		p_t(y) \geq c \frac{t}{|y|^{d+1}}, \qquad |y| \geq t,
	\end{equation*}
	for some constant $c>0$, we find from 
	\begin{equation*}
			\frac{\mbb{E} f(X_t)-f(0)}{t}
			\geq \frac{1}{t} \mbb{E} \left( \frac{|X_t|}{|\log (|X_t|)|} \I_{\{0<|X_t|<1/4\}} \right) 
		\end{equation*}
	that \begin{align*}
	\frac{\mbb{E} f(X_t)-f(0)}{t}
	\geq \frac{1}{t} \int_{0<|y|<1/4} \frac{|y|}{|\log(|y|)|} p_t(y) \, dy 
	&\geq c \int_{t<|y|<1/4} \frac{1}{|y|^d} \frac{1}{|\log(|y|)|} \, dy \\
	&\xrightarrow[]{t \to 0} \infty,
	\end{align*}
	and so $f \notin \mc{D}(A)$. 
\end{proof}

\begin{thm} \label{levy-19}
	Let $(X_t)_{t \geq 0}$ be a $2$-dimensional Cauchy process with generator $(A,\mc{D}(A))$. Then there exists a function $f \in \mc{D}(A)$ which is not Lipschitz continuous.
\end{thm}

Let us mention that the proof of Theorem~\ref{levy-19} has been inspired by G\"{u}nter \cite{guenter34} who constructed a function $f \in C_{\infty}(\mbb{R}^3)$ which is in the domain of the generator of three-dimensional Brownian motion but which is not twice differentiable, see \cite[Example 7.25]{bm2} for a modern account. \par
For the proof of Theorem~\ref{levy-19} we need an auxiliary result concerning the potential operator of an isotropic $\alpha$-stable L\'evy process $(X_t)_{t \geq 0}$. Recall that the potential operator $(R_0,\mc{D}(R_0))$ (in the sense of Yoshida) associated with a L\'evy process $(X_t)_{t \geq 0}$ and resolvent $(R_{\lambda})_{\lambda>0}$ is defined by \begin{align*}
	\mc{D}(R_0) &:= \{f \in C_{\infty}(\mbb{R}^d); \exists g \in C_{\infty}(\mbb{R}^d): \, \, \lim_{\lambda \to 0} \|R_{\lambda} f-g\|_{\infty} = 0\}, \\
	R_0 f &:= \lim_{\lambda \to 0} R_{\lambda} f, \quad f \in \mc{D}(R_0),
\end{align*}
see \cite[Section 11]{berg} for a thorough discussion.

\begin{lem} \label{levy-21}
	Let $(X_t)_{t \geq 0}$ be a $d$-dimensional isotropic $\alpha$-stable L\'evy process with resolvent $(R_{\lambda})_{\lambda>0}$. If $\alpha<d$ then there exists a finite constant $c_{d,\alpha}>0$ such that \begin{equation}
		\sup_{\lambda>0} R_{\lambda} u(x) = c_{d,\alpha}  \int_{\mbb{R}^d} |z|^{-d+\alpha} u(x+z) \, dz \label{levy-eq59}
	\end{equation}
	for any $u \in C_{\infty}(\mbb{R}^d)$, $u \geq 0$. In particular, any non-negative function $u \in C_{\infty}(\mbb{R}^d) \cap L^1(dx)$ is in the domain $\mc{D}(R_0)$ of the potential operator $R_0$.
\end{lem}

\begin{proof}[Proof of Lemma~\ref{levy-21}]
	Identity \eqref{levy-eq59} is a direct consequence of the scaling property of the transition density of $(X_t)_{t \geq 0}$; it is a classical result in potential theory, see e.\,g.\ \cite{bogdan09} for a proof. For the second assertion, we note that $\int_{\mbb{R}^d \cap B(0,1)} |y|^{\alpha-d} \, dy <\infty$ implies, by the dominated convergence theorem, that $\sup_{\lambda>0} R_{\lambda} u \in C_{\infty}(\mbb{R}^d)$ for any non-negative function $u \in C_{\infty}(\mbb{R}^d) \cap L^1(dx)$. By \cite[Theorem 7.24(d)]{bm2} this entails that $u \in \mc{D}(R_0)$ for any such function $u$.
\end{proof} 

\begin{proof}[Proof of Theorem~\ref{levy-19}]
	
	As $R_0(\mc{D}(R_0)) \subseteq \mc{D}(A)$, cf.\ \cite[Lemma 11.13(vi)]{berg}, it suffices to find $u \in \mc{D}(R_0)$ such that $R_0 u$ is not Lipschitz continuous. It follows from Lemma~\ref{levy-21} and the linearity of $R_0$ that \begin{equation*}
		R_0 u(x) = c \int_{\mbb{R}^2} |z|^{-1} u(x-z) \, dz, \qquad x \in \mbb{R}^2
	\end{equation*}
	for any function $u \in C_{\infty}(\mbb{R}^2) \cap L^1(dx)$. Pick a function $f \in C_c([0,1))$ such that $f \geq 0$ and $f(0)=0$. If we define \begin{equation*}
		u(x_1,x_2) := \frac{x_1}{\sqrt{x^2_1+x_2^2}} \frac{|x_2|}{\sqrt{x_1^2+x_2^2}} f \left(\sqrt{x_1^2+x_2^2}\right), \qquad x=(x_1,x_2) \in \mbb{R}^2, 
	\end{equation*}
	then $u \in C_{\infty}(\mbb{R}^2) \cap L^1(dx) \subseteq \mc{D}(R_0)$. We will show that $f$ can be chosen in such a way that $x \mapsto R_0 u(x)$ is not Lipschitz continuous at $x=0$. Introducing polar coordinates we find \begin{align*}
		R_0 u(0,x_2)
		= c \int_0^1 r f(r) \left( \int_0^{2\pi} \frac{|\sin \varphi| \cos \varphi}{\sqrt{x_2^2+r^2-2rx_2 \cos \varphi}} \, d\varphi \right) \, dr.
	\end{align*}
	Writing \begin{align*}
		I:=\int_0^{2\pi} \frac{|\sin \varphi| \cos \varphi}{\sqrt{x_2^2+r^2-2rx_2 \cos \varphi}} \, d\varphi = \left( \int_0^{\pi} + \int_{\pi}^{2\pi} \right) \frac{|\sin \varphi| \cos \varphi}{\sqrt{x_2^2+r^2-2rx_2 \cos \varphi}} \, d\varphi
	\end{align*}
	and performing a change of variables, $t := \sqrt{x_2^2+r^2-2rx_2 \cos \varphi}$, we get for $x_2>0$ \begin{align*}
		I 
		= \frac{1}{rx_2} \int_{|x_2-r|}^{x_2+r} \cos \varphi(t)  \, dt -  \frac{1}{rx_2} \int_{x_2+r}^{|x_2-r|} \cos \varphi(t) \, dt 
		&= \frac{2}{rx_2} \int_{|x_2-r|}^{x_2+r} \frac{x_2^2+r^2-t^2}{rx_2} \, dt \\
	\end{align*}
	and so
	\begin{equation*}
		I = \frac{4}{3} \frac{r}{x_2^2} \I_{\{r \leq x_2\}} + \frac{4}{3} \frac{x_2}{r^2} \I_{\{r>x_2\}}.
	\end{equation*}
	Hence, \begin{align*}
	R_0 u(0,x_2)
	= \frac{4}{3} c \left( \frac{1}{x_2^2} \int_0^{x_2} r^2 f(r) \, dr + x_2 \int_{x_2}^1 \frac{f(r)}{r} \, dr \right).
	\end{align*}
	As $u(x) =-u(-x)$ we have \begin{equation*}
	\int_{\mbb{R}^2} |z|^{-1} u(z) \, dz 
	= \int_{\mbb{R}^2} |z|^{-1} u(-z) \, dz
	= - \int_{\mbb{R}^2} |z|^{-1} u(z) \, dz
	\end{equation*}
	which implies $R_0 u(0)=0$. Consequently, we have shown that \begin{align*}
	\frac{R_0 u(0,x_2)-R_0u(0)}{x_2}
	= \frac{4}{3} c \left( \frac{1}{x_2^3} \int_0^{x_2} r^2 f(r) \, dr + \int_{x_2}^1 \frac{f(r)}{r} \, dr \right).
	\end{align*}
	If we choose $f(r) = |\log(r)|^{-1} \chi(r) 1_{(0,\infty)}(r)$ for a cut-off function $\chi$ satisfying $\I_{[0,1/2]} \leq \chi \leq \I_{[0,1)}$, then $\lim_{x_2 \downarrow 0} \int_{x_2}^1 \frac{f(r)}{r} \, dr = \infty$ and \begin{align*}
		\left| \frac{1}{x_2^3} \int_0^{x_2} r^2 f(r) \, dr \right|
		=  \left|\frac{1}{x_2^3} \int_0^{x_2} \frac{r^2}{|\log r|} \, dr \right| 
		\leq \frac{1}{|\log |x_2||} \frac{1}{x_2^3} \int_0^{x_2} r^2 \, dr 
	\xrightarrow[]{x_2 \downarrow 0} 0.
	\end{align*}
	Thus, \begin{equation*}
		\lim_{x_2 \downarrow 0}\frac{R_0 u(0,x_2)-R_0u(0)}{x_2}=\infty,
	\end{equation*}
	i.\,e.\ $x \mapsto R_0 u(x)$ is not Lipschitz continuous at $x=0$.
\end{proof}

\begin{ack}
	I am grateful to Niels Jacob and Ren\'e Schilling for valuable comments which helped to improve the presentation of this paper; I owe the proof of Remark~\ref{levy-5}\eqref{levy-5-ii} to Ren\'e Schilling. Moreover, I thank the \emph{Institut national des sciences appliqu\'ees de Toulouse, G\'enie math\'ematique et mod\'elisation} for its hospitality during my stay in Toulouse, where a part of this work was accomplished.
\end{ack}


\begin{thebibliography}{99}\frenchspacing	

\bibitem{bae} 
	Bae, J., Kassmann, M.: Schauder estimates in generalized H\"older spaces. Preprint arXiv 1505.05498

\bibitem{barles10} 
	Barles, G., Chasseigne, E., Imbert, C.: H\"older continuity of solutions of second-order non-linear elliptic integro-differential equations. \emph{J. Euro. Math. Soc.} \textbf{13} (2011), 1--26.

\bibitem{bass08} 
	Bass, R.F.: Regularity results for stable-like operators. \emph{J.\ Funct.\ Anal.} \textbf{259} (2009), 2693--2722.

\bibitem{berg} 
	Berg, C., Forst, G.: \emph{Potential Theory on Locally Compact Abelian Groups}. Springer Berlin Heidelberg, Berlin, Heidelberg 1975. 	

\bibitem{blumen61} 
	Blumenthal, R. M., Getoor, R. K.: Sample Functions of Stochastic Processes with Stationary Independent Increments. \emph{J.\ Math.\ Mech.} \textbf{10} (1961), 493--516.
	
\bibitem{bogdan09} 
	Bogdan, K., Byczkowski, T., Kulczycki, T., Ryznar, M., Song, R., Vondra\u{c}ek, Z.: \emph{Potential Analysis of Stable Processes and its Extensions}. Springer, Berlin 2009.

\bibitem{ltp}
   B\"{o}ttcher, B., Schilling, R.\,L., Wang, J.: \emph{L\'evy-Type Processes: Construction, Approximation and Sample Path Properties}. Springer Lecture Notes in Mathematics vol.\ \textbf{2099}, (vol.~III of the ``L\'evy Matters'' subseries). Springer, 2014.

\bibitem{bouleau91} Bouleau, N., Hirsch, F.: \emph{Dirichlet Forms and Analysis on Wiener Space}. De Gruyter, 1991.
	
\bibitem{caff}
	Caffarelli, L., Silvestre, L.: Regularity theory for fully nonlinear integro-differential equations. \emph{Comm. Pure Appl. Math.} \textbf{62} (2009), 597--638.
	
\bibitem{meyer4}
	Dellacherie, C., Meyer, P.-A.,: \emph{Théorie du potentiel associ\'ee à une r\'esolvante - th\'eorie des processus de Markov}. Hermann, Paris 1987.

\bibitem{dong12} 
	Dong, H., Kim, D.: Schauder estimates for a class of non-local elliptic equations, \emph{Discrete Contin. Dyn. Syst.} \textbf{33} (2012), 2319--2347.


\bibitem{gilbarg} 
	Gilbarg, D., Trudinger, N. S.: \emph{Elliptic Partial Differential Equations of Second Order}. Springer, Berlin 1983.
	
\bibitem{grz14} 
	Grzywny, T.: On Harnack Inequality and H\"older Regularity for Isotropic Unimodal L\'evy Processes. \emph{Potential Anal.} \textbf{41} (2014), 1--29.
	    
\bibitem{szcz17}
	 Grzywny, T., Szczypkowski, K.: Estimates of heat kernels of non-symmetric L\'evy processes. Preprint arXiv 1710.07793.

\bibitem{guenter34} 
	G\"{u}nter, N. M.: \emph{La Th\'eorie du Potentiel et ses Applications aux Probl\`emes Fondamentaux de la Physique Math\'ematique}. Gauthier-Villars, Paris 1934.
	
\bibitem{hansen17} 
	Hansen, W.: Intrinsic H\"{o}lder Continuity of Harmonic Functions. \emph{Potential Anal.} \textbf{47} (2017), 1--12.
	
		  	\bibitem{jac2}
		  		Jacob, N.: \emph{Pseudo Differential Operators and Markov Processes II}. Imperial College Press/World Scientific, London 2002.

	\bibitem{kaleta15} 
		Kaleta, K., Sztonyk, P.: Estimates of transition densities and their derivatives for jump L\'evy processes. \emph{J.\ Math.\ Anal.\ Appl.} \textbf{431} (2015), 260--282.
	

\bibitem{knop13} 
	Knopova, V., Schilling, R. L.: A note on the existence of transition probability densities of L\'evy processes. \emph{Forum. Math.} \textbf{25} (2013), 125--149. 

	\bibitem{moments} 
		K\"{u}hn, F.: Existence and estimates of moments for Lévy-type processes. \emph{Stoch. Proc. Appl.} \textbf{127} (2017), 1018--1041.

	\bibitem{matters}
		K{\"u}hn, F.: \emph{L\'evy-Type Processes: Moments, Construction and Heat Kernel Estimates}. Springer Lecture Notes in Mathematics vol.\ \textbf{2187} (vol.~VI of the ``L\'evy Matters'' subseries). Springer, 2017.
		
	\bibitem{reg-feller}
		K\"{u}hn, F.: Schauder estimates for Poisson equations associated with non-local Feller generators. Preprint arXiv 1902.01760. 

	\bibitem{schoenberg} 
		K\"{u}hn, F., Schilling, R.L.: A probabilistic proof of Schoenberg's theorem.  To appear: \emph{J.\ Math.\ Anal.\ Appl.} DOI: 10.1016/j.jmaa.2018.11.046.

	\bibitem{euler-maruyama} 
		K\"{u}hn, F., Schilling, R.L.: Strong convergence of the Euler–Maruyama approximation for a class of L\'evy-driven SDEs. To appear: \emph{Stoch. Proc. Appl.} DOI: 10.1016/j.spa.2018.07.018.
	\bibitem{ihke}
		K\"{u}hn, F., Schilling, R.\,L.: On the domain of fractional Laplacians and related generators of Feller processes. To appear: \emph{J.\ Funct.\ Anal.} DOI: 10.1016/j.jfa.2018.12.011.
	\bibitem{ryznar15} 
		Kulczycki, T., Ryznar, M.: Gradient estimates of harmonic functions and transition densities for L\'evy processes. \emph{Trans. Amer. Math. Soc.} \textbf{368} (2015), 281--318.
	
	\bibitem{kwas17} 
		Kwa\'snicki, M.: Ten equivalent definitions of the fractional Laplace operator. \emph{Fract.\ Calc.\ Anal.\ Appl.} \textbf{20} (2017), 7--51.
	
\bibitem{lunardi} Lunardi, A.: \emph{Interpolation Theory}. Scuola Normale Superiore, Pisa 2009.

\bibitem{ros-oton16} 
	Ros-Oton, X., Serra, J.: Regularity theory for general stable operators. \emph{J. Diff. Equations} \textbf{260} (2016), 8675--8715.


\bibitem{sato} 
	Sato, K.: \emph{L\'evy processes and infinitely divisible distributions}. Cambridge University Press, Cambridge 2013.

\bibitem{rs-growth} 
	Schilling, R. L.: Growth and H\"older conditions for the sample paths of Feller processes. \emph{Probab. Theory Related Fields} \textbf{112} (1998), 565--611.

\bibitem{mims} 
	Schilling, R. L.: \emph{Measures, Integrals and Martingales}. Cambridge University Press, 2017 (2nd edition).

\bibitem{bm2}
	Schilling, R.\,L., Partzsch, L.: \emph{Brownian Motion. An Introduction to Stochastic Processes}. De Gruyter, Berlin 2014 (2nd ed).


\bibitem{bernstein}
	Schilling, R. L., Song, R., Vondra\u{c}ek, Z.: \emph{Bernstein functions: theory and applications}. De Gruyter, Berlin 2012 (2nd ed).

\bibitem{ssw12} 
	Schilling, R.L., Sztonyk, P., Wang, J.: Coupling property and gradient estimates of Lévy processes via the symbol. \emph{Bernoulli} \textbf{18} (2012), 1128--1149.
	
\bibitem{stein} 
	Stein, E. M.: \emph{Singular integrals and differentiability properties of functions}. Princeton Univ. Press,  1970.
	
	\bibitem{stein93} Stein, E. M.: \emph{Harmonic Analysis}. Princeton University Press, 1993. 
	
	
	\bibitem{szt10} 
		Sztonyk, P.: Regularity of harmonic functions for anisotropic fractional Laplacians. \emph{Math. Nachr.} \textbf{283} (2010), 289--311.
	
	
\bibitem{triebel78} 
	Triebel, H.: \emph{Interpolation theory, function spaces, differential operators}. North-Holland Pub. Co, 1978.



\end{thebibliography}
\end{document}